\documentclass{amsart}
\usepackage{graphicx,amsmath,amsthm,amssymb, amsfonts}
\usepackage[all]{xy}
\newtheorem{definition}{Definitions}[section]

\newtheorem{lemma}[definition]{Lemma}
\newtheorem{prop}[definition]{Proposition}
\newtheorem{theo}[definition]{Theorem}
\newtheorem{Coro}[definition]{Corollary}
\newtheorem{remark}[definition]{Remark}

\newcommand{\Ext}{{\rm Ext}}
\newcommand{\Tor}{{\rm Tor}}

\newcommand{\im}{{\rm Im}}

\newcommand{\coker}{{\rm coker}}
\newcommand{\Hom}{{\rm Hom}}

\newcommand{\Ass}{{\rm Ass}}

\newcommand{\Supp}{{\rm Supp}}

\newcommand{\Max}{{\rm Max}}

\newcommand{\fp}{{\frak p}}

\newcommand{\fm}{{\frak m}}

\newcommand{\fa}{{\frak a}}
\newcommand{\fb}{{\frak b}}

\newcommand{\dlim}{{\displaystyle\lim_{\stackrel{\longrightarrow}{\scriptscriptstyle{n\in\Bbb{N}}}}}}
\begin{document}

\title[Hartshorne's questions and weakly cofiniteness]
{Hartshorne's questions and weakly cofiniteness}%

\author[Roshan]{Hajar Roshan-Shekalgourabi$^*$}%
\address{Department of Basic Sciences, Arak University of Technology, P. O. Box 38135-1177, Arak,  Iran.}%
\email{hrsmath@gmail.com and Roshan@arakut.ac.ir}%

\author[Hatamkhani]{Marzieh Hatamkhani}%
\address{Department of Mathematics, Faculty of Science, Arak University, Arak, 38156-8-8349, Iran.}%
\email{m-hatamkhani@araku.ac.ir}%

\thanks{$^*$ Corresponding author}%
\subjclass[2010]{13D45, 13E05, 18E10}%
\keywords{weakly Laskerian modules, weakly cofinite modules, Krull dimension, Local cohomology modules, Abelian category.}%

\date{\today}%
\begin{abstract}
Let $R$ be a commutative Noetherian ring, $\fa$ be an ideal of $R$ and $M$ be an $R$-module.  The main purpose of this paper is to answer the Hartshorn's questions in the class of weakly Laskerian modules. It is shown that if $s\geq 1$  is a positive integer such that $\Ext^j_R(R/\fa, M)$ is weakly Laskerian for all $j\leq s$ and the $R$-module $H^i_\fa(M)$ is $FD_{\leq 1}$ for all $i < s$, then the $R$-module $H^i_\fa(M)$ is $\fa$-weakly cofinite for all $i <s$. In addition, we show that the category of all $\fa$-weakly cofinite $FD_{\leq 1}$ $R$-modules is an Abelian subcategory of the category of all $R$-modules. Also, we prove that if $\Ext^i_R(R/\fa,M)$ is weakly Laskerian for all $i\leq \dim M$, then the $R$-module $\Ext^i_R(N,M)$ is weakly Laskerian for all $i\geq 0$ and for any finitely generated $R$-module $N$ with $\Supp_R(N) \subseteq V (\fa)$ and $\dim N \leq 1$.
\end{abstract}
\maketitle

\section{INTRODUCTION}
Let $R$ denote a commutative Noetherian ring with identity and $\fa$ be an ideal of $R$. For an $R$-module $M$, the $i$th local cohomology module of $M$ with respect to $\fa$ is defined as $$H^i_\fa(M)\cong\dlim \Ext^i_R(R/\fa^n, M).$$ For more details about the local cohomology, we refer the reader to \cite{BS}.

In 1968, Grothendieck \cite{Gro69} conjectured that for any ideal $\fa$ of $R$ and any finitely generated $R$-module $M$, $\Hom_R \left(R/\fa, H^i_\fa (M)\right)$ is a finitely generated $R$-module for all $i$. One year later, by proving a counterexample, Hartshorne \cite{Ha} showed that the Grothendieck's conjecture is not true in general even $R$ is regular and introduced the class of cofinite modules with respect to an ideal. He defined an $R$-module $M$ to be \emph{$\fa$-cofinite} if $\Supp_R(M) \subseteq V(\fa)$ and $\Ext^j_R (R/\fa,M)$ is finitely generated for all $j$ and posed the following questions:
\begin{enumerate}
  \item For which rings $R$ and ideals $\fa$ is the module $H^i_\fa (M)$ $\fa$-cofinite for all $i$ and all finitely generated $R$-modules $M$?\label{Q1}
  \item Is the category of $\fa$-cofinite modules an Abelian subcategory of the category of all $R$-modules? That is, if $f : M\rightarrow N$ is an $R$-homomorphism of $\fa$-cofinite modules, are $\ker f$ and $\coker f$ $\fa$-cofinite?\label{Q2}
\end{enumerate}

There are many papers that are devoted to study these questions. For example, with respect to the question (\ref{Q1}), see \cite{Ha, Chir, DM, Y, BN} and with respect to the question (\ref{Q2}), see \cite{Kawa2011, Mel2012, BNS2014, BNSwl}. Recently, Aghapournahr and Bahmanpour in \cite{AB} introduced the class of $FD_{\leq n}$ where $n \geq -1$ is an integer. An $R$-module $M$ is said to be ${FD}_{\leq n}$ if there is a finitely generated submodule $N$ of $M$ such that $\dim M/N\leq n$.
As an extension of the above results, they proved in \cite{AB} that if $M$ is a finitely generated $R$-module such that $H^i_\fa(M)$ is ${FD}_{\leq 1}$ for all $i$, then the $R$-module $H^i_\fa(M)$ is $\fa$-cofinite for all $i$. They also showed that the category of $\fa$-cofinite ${FD}_{\leq 1}$ $R$-modules is an Abelian subcategory of the category of all $R$-modules.

Based on \cite{DiMa} and \cite{DiM}, $M$ is called \emph{weakly Laskerian} if $Ass_R(M/N)$ is a finite set for each submodule $N$ of $M$. Also, $M$ is said to be \emph{$\fa$-weakly cofinite} if $\Supp_R(M)\subseteq V (\fa)$ and $\Ext^i_R (R/\fa, M )$ is weakly Laskerian, for all $i \geq 0$. In \cite{FSF}, Quy  has introduced the class of \emph{FSF} modules, modules containing some finitely generated submodules such that the support of the quotient module is finite.
It has shown in \cite[Theorem 3.3]{Bahmanpour} that over a Noetherian ring $R$, an $R$-module $M$ is weakly Laskerian if and only if it is FSF.
Since the concept of weakly Laskerian modules is a natural generalization of the concept of finitely generated modules, many authors studied the weakly Laskerianness  of local cohomology modules and answered the Hartshorn's questions in the class of weakly Laskerian modules (see \cite{DiM, DiMa, AB2014}). More recently, Bahmanpour et. al. in \cite{BNSwl} showed that the category of all $\fa$-weakly cofinite $R$-modules $M$, with $\dim M \leq 1$ forms an Abelian category.

The main purpose of this paper is to answer the Hartshorn's questions in the class of weakly Laskerian modules and generalize the above mentioned results. In this direction, in Section 3, we prove the following result.

\begin{theo}\label{In1}
Let  $M$ be an $R$-module and $s\geq 1$  be a positive integer such that $\Ext^j_R(R/\fa, M)$ is weakly Laskerian for all $j\leq s$ and the $R$-module $H^i_\fa(M)$ is $FD_{\leq 1}$ for all $i < s$. Then the following statements hold:
\begin{enumerate}
  \item The $R$-module $H^i_\fa(M)$ is $\fa$-weakly cofinite for all $i <s$.
  \item For every $FD_{\leq 0}$ submodule $X$ of $H^s_\fa(M)$, the $R$-module $\Ext^i_R(R/\fa, H^s_\fa(M)/X)$ is weakly Laskerian for $i=0,1$. In particular, the set $\Ass_R(H^s_\fa(M)/X)$ is finite.
\end{enumerate}
\end{theo}

We also prove the category of all $\fa$-weakly cofinite $FD_{\leq 1}$ $R$-modules is an Abelian subcategory of the category of all $R$-modules. The proof of this result is given in Theorem \ref{n5}.
Our main tools for proving these results is the following, which is an extension of \cite[Proposition 3.2]{BNSwl}.

\begin{prop}\label{In2}
Let $M$ be a non-zero $R$-module (not necessary $\fa$-torsion) such that $\dim M\leq 1$. Then the following conditions are equivalent:
\begin{enumerate}
  \item $H^i_\fa (M)$ is $\fa$-weakly cofinite for all $i\geq 0$;
  \item The $R$-module $\Ext^i_R(R/\fa,M)$ is weakly Laskerian for all $i\geq 0$;
  \item The $R$-modules $\Hom_R(R/\fa,M)$ and $\Ext^1_R(R/\fa,M)$ are weakly Laskerian.
\end{enumerate}
\end{prop}

In the sequel, we will  state some conditions for the weakly cofiniteness of local cohomology modules with respect to ideals of dimension at most one.  More precisely, we prove the following theorem:

\begin{theo}\label{In7}
Let $M$ be an $R$-module such that $\Ext^i_R(R/\fa,M)$ is weakly Laskerian for all $i\leq \dim M$. Then the following assertions hold:
\begin{enumerate}
  \item The $R$-module $H^i_\fb(M)$ is $\fb$-weakly cofinite for all $i\geq 0$ and for any ideal $\fb\subseteq \fa$ with $\dim R/\fb\leq1$.
  \item The $R$-module $\Ext^i_R(N,M)$ is weakly Laskerian for all $i\geq 0$ and for any finitely generated $R$-module $N$ with $\Supp_R(N) \subseteq V (\fa)$ and $\dim N \leq 1$.
\end{enumerate}
\end{theo}

The proof of Theorem \ref{In7} is given in Proposition \ref{n6} and Theorem \ref{n7}.

Throughout the paper, we assume that $R$ is a commutative Noetherian ring, $\fa$ is an ideal of $R$ and $V(\fa)$ is the set of all prime ideals of $R$ containing $\fa$. For any unexplained notation and terminology we refer the reader to \cite{Matsu}.

\section{PRELIMINARIES}

Recall that a class of $R$-modules is a \emph{Serre subcategory} of the category of $R$-modules when it is closed under taking submodules, quotients and extensions. For example, the classes of Noetherian modules, Artinian modules and weakly Laskerian modules are Serre subcategories. As in standard notation, we let $\mathcal{S}$ stand for a Serre subcategory of the category of $R$-modules.
The following lemma which is needed in the next sections, immediately follows from the definition of $\Ext$ and $\Tor$ functors.

\begin{lemma}\label{pro1}
Let $M$ be a finitely generated $R$-module and $N\in \mathcal{S}$. Then $\Ext^i_
R(M,N)\in \mathcal{S}$ and $\Tor_i^R (M,N)\in \mathcal{S}$ for all $i\geq0$.
\end{lemma}

\begin{lemma}\label{Gruson}
Suppose that $M$ is a finitely generated $R$-module and $N$ is an arbitrary $R$-module. Let for some $t\geq0$, $\Ext^i_R(M,N)\in \mathcal{S}$ for all $i\leq t$. Then $\Ext^i_R(L,N)\in \mathcal{S}$ for all $i\leq t$ and any finitely generated $R$-module $L$ with $\Supp_R(L)\subseteq\Supp_R(M)$.
\end{lemma}
\begin{proof}
See \cite[Lemma 2.2]{AR}.
\end{proof}

Let us mention some elementary properties of the weakly Laskerian modules that we shall use.

\begin{remark}\label{pro}
The following statements hold:
\begin{enumerate}
  \item The class of weakly Laskerian modules contains all minimax modules. In particular, this class contains all finitely generated and all Artinian modules.\label{i}
  \item Let $0\rightarrow L\rightarrow M\rightarrow N\rightarrow0$ be an exact sequence of $R$-modules. Then $M$ is weakly Laskerian if and only if $L$ and $N$ are both weakly Laskerian (see \cite[Lemma~2.3]{DiMa}). Thus any submodule and quotient of a weakly Laskerian module is weakly Laskerian.\label{ii}
  \item Based on \cite{w.a.hajikarimi}, an $R$-module $M$ is said to be weakly Artinian if $E_R(M)$, its injective envelope, can be written as $E_R(M) :=\oplus_{i=1}^n \mu^0(\fm_i,M)E_R(R/\fm_i)$ where $\fm_1, \cdots, \fm_n$ are maximal ideals of $R$. By \cite[Lemma 2.3]{w.a.hajikarimi}, an $R$-module $M$ is weakly Artinian if and only if $M$ is weakly Laskerian and $\Ass_R(M) \subseteq \Max(R)$.\label{iii}
\end{enumerate}
\end{remark}

\begin{lemma}\label{wA}
Let $M$ be an $\fa$-torsion $R$-module. If $(0 :_M \fa)$ is a weakly Laskerian $R$-module with support in $\Max(R)$, then $M$ is also weakly Laskerian.
\end{lemma}
\begin{proof}
The assertion follows from Remark \ref{pro}(\ref{F3}),  \cite[Lemma~2.8.]{w.a.hajikarimi} and the fact that $$\Ass_R(0 :_M \fa)= \Ass_R(M)\cap V(\fa)= \Ass_R(M).$$
\end{proof}

\begin{lemma}\label{minimax}
Let $\fa$ be an ideal of $R$, $M$ be an $R$-module and $n$ be a non-negative integer such that $\Ext^n_R(R/\fa, M)$ (resp. $\Ext^{n+1}_R(R/\fa, M)$) is in $\mathcal{S}$. If  $\Ext^j_R(R/\fa, H^i_\fa(M))$ is in $\mathcal{S}$ for all $j$ and all $i<n$, then $\Hom_R(R/\fa, H^n_\fa(M))$ (resp. $\Ext^1_R(R/\fa, H^n_\fa(M))$) is in $\mathcal{S}$.
\end{lemma}
\begin{proof}
See \cite[Lemma~2.3]{AbaB2015}.
\end{proof}

\section{MAIN RESULTS}
Let  $n \geq -1$ be  an integer. Recall that an $R$-module $M$ is said to be $FD_{\leq n}$ if there is a finitely generated submodule $N$ of $M$ such that $\dim M/N\leq n$. The concept of $FD_{\leq n}$ modules introduced by  Aghapournahr and  Bahmanpour \cite{AB} as an interesting example of the class of extension modules introduced by Yoshizawa \cite{Yoshizawa}. By definition, any finitely generated $R$-module and any $R$-module with dimension at most $n$ is $FD_{\leq n}$. The class of all $FD_{\leq n}$ $R$-modules forms a Serre subcategory of the category of all $R$-modules by \cite[Lemma~2.3]{AB2014}.

As the first main result of this paper, we are going to prove the following theorem which states some conditions for the weakly cofiniteness of local cohomology modules.

\begin{theo}\label{n4}
Let  $M$ be an $R$-module and $s\geq 1$  be a positive integer such that $\Ext^j_R(R/\fa, M)$ is weakly Laskerian for all $j\leq s$ and the $R$-module $H^i_\fa(M)$ is $FD_{\leq 1}$ for all $i < s$. Then the following statements hold:
\begin{enumerate}
  \item The $R$-module $H^i_\fa(M)$ is $\fa$-weakly cofinite for all $i <s$.\label{n41}
  \item For every $FD_{\leq 0}$ submodule $X$ of $H^s_\fa(M)$, the $R$-module $\Ext^i_R(R/\fa, H^s_\fa(M)/X)$ is weakly Laskerian for $i=0,1$. In particular, the set $\Ass_R(H^s_\fa(M)/X)$ is finite.\label{n42}
\end{enumerate}
\end{theo}

We divide the proof of Theorem \ref{n4} into a sequence of lemmas and propositions.

\begin{lemma}\label{FD0}
Let $M$ be an $\fa$-torsion $FD_{\leq 0}$ $R$-module. Then the following statements are equivalent:
\begin{enumerate}
  \item $M$ is weakly Laskerian.\label{F1}
  \item $M$ is $\fa$-weakly cofinite. \label{F2}
  \item The $R$-module $\Hom_R(R/\fa,M)$ is weakly Laskerian.\label{F3}
\end{enumerate}
\end{lemma}
\begin{proof}
($\ref{F1})\Rightarrow (\ref{F2}$) and ($\ref{F2})\Rightarrow (\ref{F3}$) are clear. For ($\ref{F3})\Rightarrow (\ref{F1}$), by definition we have the long exact sequence $$0\rightarrow \Hom(R/\fa, F)\rightarrow \Hom(R/\fa, M)\rightarrow \Hom(R/\fa, D)\rightarrow \Ext^1_R(R/\fa, F)\rightarrow \cdots$$ where $F$ is finitely generated and $D$ is an $R$-module with $\dim D\leq 0$. Thus, by assumption the $R$-module $\Hom(R/\fa, D)$ is a weakly Laskerian $R$-module with support in $\Max(R)$. Hence, the assertion follows from Lemma \ref{wA}.
\end{proof}

\begin{lemma}\label{n1}
Let $M$ be an $\fa$-torsion $R$-module such that $\dim M\leq 1$. Then $M$ is $\fa$-weakly cofinite if and only if the $R$-modules $\Hom_R(R/\fa,M)$ and $\Ext^1_R(R/\fa,M)$ are weakly Laskerian.
\end{lemma}
\begin{proof}
See \cite[Proposition 3.2]{BNSwl}.
\end{proof}

In the following proposition that is a generalization of \cite[Proposition 3.2]{BNSwl}, we prove the assertion of Lemma \ref{n1}  for any $R$-module $M$ with $\dim M\leq 1$ not necessarily $\fa$-torsion.

\begin{prop}\label{n2}
Let $M$ be a non-zero $R$-module (not necessary $\fa$-torsion) such that $\dim M\leq 1$. Then the following conditions are equivalent:
\begin{enumerate}
  \item $H^i_\fa (M)$ is $\fa$-weakly cofinite for all $i\geq 0$;\label{n21}
  \item The $R$-module $\Ext^i_R(R/\fa,M)$ is weakly Laskerian for all $i\geq 0$;\label{n22}
  \item The $R$-modules $\Hom_R(R/\fa,M)$ and $\Ext^1_R(R/\fa,M)$ are weakly Laskerian.\label{n23}
\end{enumerate}
\end{prop}

\begin{proof}
(\ref{n21}) $\Rightarrow$ (\ref{n22}) follows from \cite[Proposition~3.9]{Mel}.

(\ref{n22}) $\Rightarrow$ (\ref{n23}) is clear.

(\ref{n23}) $\Rightarrow$ (\ref{n21}): By Grothendieck's Vanishing Theorem \cite[Theorem~6.1.2]{BS}, we only need to show that $\Gamma_\fa (M)$ and $H^1_\fa(M)$ are $\fa$-weakly cofinite. To do this, consider the exact sequence $$0\rightarrow \Gamma_\fa (M)\rightarrow M \rightarrow  M/\Gamma_\fa (M)\rightarrow  0$$
which induces the exact sequence
\begin{align*}
    0&\rightarrow \Hom_R(R/\fa, \Gamma_\fa (M))\rightarrow \Hom_R(R/\fa, M)\rightarrow \Hom_R(R/\fa, M/\Gamma_\fa (M))\\ &\rightarrow \Ext^1_R(R/\fa, \Gamma_\fa (M)) \rightarrow \Ext^1_R(R/\fa, M)\rightarrow \cdots.
  \end{align*}

Hence, as $\Hom_R(R/\fa, M/\Gamma_\fa (M))=0$, we infer that $\Hom_R(R/\fa, \Gamma_\fa (M))$ and $\Ext^1_R(R/\fa, \Gamma_\fa (M))$ are weakly Laskerian $R$-modules by assumption. Thus $\Gamma_\fa (M)$ is $\fa$-weakly cofinite by Lemma \ref{n1}. This enable us to deduce that  $\Hom_R(R/\fa, H^1_\fa(M))$ is weakly Laskerian by assumption and Lemma \ref{minimax}. Now, let $\fp \in \Supp_R(H^1_\fa(M))$. Then $\fp \in \Supp_R(M)$ and $(H^1_\fa(M))_\fp\neq 0$.
Since $\dim M\leq 1$, we have either $\dim R/\fp=0$ or $\dim R/\fp=1$. If $\dim R/\fp=1$, then $\fp$ is a minimal element of $\Supp_R(M)$ and so $\dim M_\fp=0$. Thus $(H^1_\fa(M))_\fp = 0$ by Grothendieck's Vanishing Theorem, which is impossible. Therefore, $\dim R/\fp=0$ and so $\fp$ is a maximal ideal of $R$. This implies that $\Hom_R(R/\fa, H^1_\fa(M))$ is a weakly Laskerian $R$-module with support in $\Max(R)$. Hence, $H^1_\fa(M)$ is weakly Laskerian by Lemma \ref{wA}. This completes the proof.
\end{proof}




\begin{prop}\label{n3}
Let $M$ be an $FD_{\leq 1}$ $R$-module. Then $\Ext^i_R(R/\fa,M)$ is weakly Laskerian for all $i\geq0$ if and only if $\Hom_R(R/\fa,M)$ and $\Ext^1_R(R/\fa,M)$ are weakly Laskerian.
\end{prop}

\begin{proof}
The sufficiency is clear. For the necessity,  by definition, there exists an exact sequence  $0\rightarrow F\rightarrow M\rightarrow D\rightarrow 0$ of $R$-modules where $F$ is finitely generated and $\dim D\leq 1$.  This induces the long exact sequence
\begin{multline*}
    0\rightarrow \Hom_R(R/\fa,F)\rightarrow \Hom_R(R/\fa,M)\rightarrow \Hom_R(R/\fa,D)\rightarrow \Ext^1_R(R/\fa,F) \\
    \rightarrow \Ext^1_R(R/\fa,M)\rightarrow \Ext^1_R(R/\fa,D)\rightarrow \Ext^2_R(R/\fa,F)\rightarrow \cdots
\end{multline*}
which implies that $\Hom_R(R/\fa, D)$ and $\Ext^1(R/\fa, D)$ are weakly Laskerian.  Thus $\Ext^i_R(R/\fa, D)$ is weakly Laskerian for all $i\geq0$ by Theorem \ref{n2}. Consequently, $\Ext^i_R(R/\fa, M)$ is weakly Laskerian for all $i\geq0$, as desired.
\end{proof}

Now, we are in the position to state the proof of Theorem \ref{n4}.\\

\noindent \textbf{Proof of Theorem \ref{n4}:}
(\ref{n41}) We prove the assertion by induction on $s$. For $s = 1$, by assumption, $\Gamma_\fa(M)$ is $FD_{\leq 1}$ and $\Hom_R(R/\fa, \Gamma_\fa(M))=\Hom_R(R/\fa, M)$ is weakly Laskerian. So, in view of Proposition \ref{n3}, it is sufficient to prove that $\Ext^1_R(R/\fa,\Gamma_\fa(M))$ is weakly Laskerian. Considering the exact sequence $$0\rightarrow \Gamma_\fa(M)\rightarrow M\rightarrow M/\Gamma_\fa(M)\rightarrow 0$$ and the fact that $\Hom_R(R/\fa,M/\Gamma_\fa(M))=0$, we get the exact sequence $$0\rightarrow \Ext^1_R(R/\fa, \Gamma_\fa(M)) \rightarrow \Ext^1_R(R/\fa, M)\rightarrow \cdots.$$ Therefore, $\Ext^1_R(R/\fa,\Gamma_\fa(M))$ is weakly Laskerian by assumption. Now, assume that $s > 1$ and the result has been proved for all $i<s$. By the inductive hypothesis, $H^i_\fa(M)$ is $\fa$-weakly cofinite for all $i < s-1$. Hence, $\Ext^i_R(R/\fa,H^{s-1}_\fa(M))$ is weakly Laskerian for $i=0,1$, by assumption and Lemma \ref{minimax}. Since $H^{s-1}_\fa(M)$ is $FD_{\leq 1}$, we infer that it is $\fa$-weakly cofinite by Proposition \ref{n3}. This completes the inductive steps.

(\ref{n42}) In view of (\ref{n41}) and Lemma \ref{minimax}, the $R$-modules $\Hom_R(R/\fa, H^s_\fa(M))$ and $\Ext^1_R(R/\fa,H^s_\fa(M))$ are weakly Laskerian. Now, consider the exact sequence $$0\rightarrow X\rightarrow H^s_\fa(M)\rightarrow H^s_\fa(M)/X \rightarrow 0.$$ Thus, $\Hom_R(R/\fa, X)$ is weakly Laskerian and so $X$ is $\fa$-weakly cofinite by assumption and Lemma \ref{FD0}. Moreover, we obtain the following exact sequence:
\begin{multline*}
  \cdots \rightarrow \Hom_R(R/\fa, H^s_\fa(M))\rightarrow \Hom_R(R/\fa, H^s_\fa(M)/X)\rightarrow \Ext^1_R(R/\fa,X) \\ \rightarrow \Ext^1_R(R/\fa,H^s_\fa(M))\rightarrow \Ext^1_R(R/\fa, H^s_\fa(M)/X)\rightarrow \Ext^2_R(R/\fa, X)\rightarrow \cdots.
\end{multline*}

Therefore, $\Hom_R(R/\fa, H^s_\fa(M)/X)$ and $\Ext^1_R(R/\fa, H^s_\fa(M)/X)$ are weakly Laskerian, as required.
The final assertion follows from Remark \ref{pro}(\ref{iii}) and the fact that $$\Ass_R(\Hom_R(R/\fa, H^s_\fa(M)/X))=\Ass_R(H^s_\fa(M)/X).$$$\hfill\Box$

As the second main result of this paper, we obtain the following theorem which extends the main result of \cite{Kawa2011}, \cite[Theorem 2.7]{BNS2014}, \cite[Theorem 3.7]{AB2014}, \cite[Theorem~2.5]{Irani}
and \cite[Proposition~3.2]{BNSwl}. For abbreviation, we say that an $R$-module $M$ (not necessary $\fa$-torsion) is $\fa$-ETH-weakly cofinite if the $R$-module $\Ext^i_R(R/\fa,M)$ is weakly Laskerian for all $i$.

\begin{theo}\label{n5}
Let $\mathcal{C}$ denote the category of all $\fa$-ETH-weakly cofinite $FD_{\leq1}$ $R$-modules. Then $\mathcal{C}$ is an Abelian category. In particular, the category of all $\fa$-weakly cofinite $FD_{\leq1}$ $R$-modules is an Abelian category.
\end{theo}

\begin{proof}
Let $M$ and $N$ be two $R$-modules belong to $\mathcal{C}$ and $f : M \rightarrow
 N$ be an $R$-homomorphism. If we prove that the $R$-modules $\ker f$ and $\coker f$ are $\fa$-ETH-weakly cofinite, the assertion follows. To do this, considering the exact sequence  $$0 \rightarrow \ker f \rightarrow M \rightarrow \im f \rightarrow 0,$$ we obtain the exact sequence
 \begin{align*}
  0 &\rightarrow \Hom_R(R/\fa, \ker f)\rightarrow \Hom_R(R/\fa, M)\rightarrow \Hom_R(R/\fa, \im f) \\ & \rightarrow \Ext^1_R(R/\fa, \ker f)\rightarrow \Ext^1_R(R/\fa, M)\rightarrow \cdots ,
\end{align*}
which follows that $\Hom_R(R/\fa, \ker f)$ and $\Ext^1_R(R/\fa, \ker f)$ are weakly Laskerian. Note that $\Hom_R(R/\fa, \im f)\subseteq \Hom_R(R/\fa, N)$ is weakly Laskerian. Therefore, we infer from Proposition \ref{n3} that $\ker f$ is $\fa$-ETH-weakly cofinite. Now, in view of the exact sequences
 $$0\rightarrow \ker f\rightarrow M\rightarrow \im f\rightarrow 0$$ and $$0\rightarrow \im f\rightarrow N\rightarrow \coker f \rightarrow 0$$ the $R$-module $\coker f$ is $\fa$-ETH-weakly cofinite, as desired.
\end{proof}

As an immediate consequence of Theorem \ref{n5} we obtain the following corollary.

\begin{Coro}
If $M$ is an $\fa$-weakly cofinite $FD_{\leq1}$ $R$-module, then $\Ext^i_R(N,M)$ and $\Tor^R_i(N,M)$ are $\fa$-weakly cofinite $FD_{\leq1}$ $R$-modules, for all finitely generated $R$-modules $N$ and all integers $i\geq 0$.
\end{Coro}

\begin{proof}
Since $N$ is finitely generated, it follows that $N$ has a free resolution of finitely generated
free modules. Now the assertion follows using Theorem \ref{n5} and computing the modules
$\Ext^i_R(N,M)$ and $\Tor^R_i(N,M)$, by this free resolution.
\end{proof}

In the sequel,  we will prove some assertions about the weakly cofiniteness of local cohomology modules with respect to ideals of dimension at most one.

\begin{prop}\label{n6}
Let $M$ be an $R$-module of dimension $n$ such that $\Ext^j_R(R/\fa,M)$ is weakly Laskerian for all $j\leq n$. Then the $R$-module $H^i_\fb(M)$ is $\fb$-weakly cofinite for all $i\geq 0$ and for any ideal $\fa\subseteq \fb$ with $\dim R/\fb\leq1$.
\end{prop}
\begin{proof}
By Grothendieck's Vanishing Theorem we only need to prove the assertion for $0\leq i \leq n$.
Let $\fb$ be an arbitrary ideal of $R$ containing $\fa$ with $\dim R/\fb\leq 1$. Then by assumption and Lemma \ref{Gruson}, $\Ext^j_R(R/\fb,M)$ is a weakly Laskerian $R$-module for all $j\leq n$.
We first prove the assertion for the case $n=0$. Then by assumption, the $R$-module $$\Hom_R(R/\fb, \Gamma_\fb(M))=\Hom_R(R/\fb, M)$$ is weakly Laskerian. Hence, $\Gamma_\fb(M)$ is weakly Laskerian (and so is weakly cofinite) by virtue of Lemma \ref{wA} and the fact that $\Supp_R(\Gamma_\fb(M))\subseteq V(\fb)$. Thus, it remains to give the proof for the case $n>0$. For this purpose, there are two cases to consider: $\dim R/\fb=0$ or $\dim R/\fb=1$.

\textbf{Case 1:} If $\dim R/\fb=0$, then in the light of assumption, $\Hom_R(R/\fb, \Gamma_\fb(M))=\Hom_R(R/\fb, M)$ is a weakly Laskerian $R$-module with support in $\Max(R)$. Hence,  $\Gamma_\fb(M)$ is weakly Laskerian by Lemma \ref{wA} and so is $\fb$-weakly cofinite. Now suppose, inductively, that $0 < i \leq n$ and the $R$-modules $$H^0_\fb(M), H^1_\fb(M), \cdots, H^{i-1}_\fb(M)$$ are $\fb$-weakly cofinite. Since $\Supp_R(H^i_\fb(M))\subseteq V(\fb)$ and the $R$-module $\Ext^j_R(R/\fb, M)$ is weakly Laskerian for all $j\leq n$, we infer from Lemma \ref{minimax} that $\Hom_R(R/\fb, H^i_\fb(M))$ is a zero-dimensional weakly Laskerian $R$-module and so $H^i_\fb(M)$ is weakly Laskerian by Lemma \ref{wA}, as desired.

\textbf{Case 2:} Let $\dim R/\fb=1$. The proof is by induction on $0 \leq i < n$. Since $\Hom_R(R/\fb, M/\Gamma_\fb(M))=0$, it follows from the assumption and the exact sequence
\begin{align*}
  &0 \rightarrow \Hom_R(R/\fb, \Gamma_\fb(M))\rightarrow \Hom_R(R/\fb,M)\rightarrow \Hom_R(R/\fb, M/\Gamma_\fb(M)) \\ & \rightarrow \Ext^1_R(R/\fb, \Gamma_\fb(M))\rightarrow \Ext^1_R(R/\fb,M)
\end{align*}
that  the $R$-modules $\Hom_R(R/\fb, \Gamma_\fb(M))$ and $\Ext^1_R(R/\fb, \Gamma_\fb(M))$ are weakly Laskerian. Hence, as $\dim \Gamma_\fb(M)\leq 1$, the $R$-module  $\Gamma_\fb(M)$ is $\fb$-weakly cofinite by Lemma \ref{n1}.
Now suppose that the assertion holds for $i-1$; we will prove it for $i$. By the inductive hypotheses, the $R$-modules $$H^0_\fb(M), H^1_\fb(M), \cdots, H^{i-1}_\fb(M)$$ are $\fb$-weakly cofinite. Since the $R$-modules $\Ext^i_R(R/\fb,M)$ and $\Ext^{i+1}_R(R/\fb,M)$ are weakly Laskerian, it follows from Lemma \ref{minimax} that the $R$-modules $$\Hom_R(R/\fb, H^i_\fb(M))\ \text{and}\  \Ext^1_R(R/\fb, H^i_\fb(M))$$  are weakly Laskerian and so in view of Lemma \ref{n1} the $R$-module $H^i_\fb(M)$ is $\fb$-weakly cofinite, for all $i = 0, 1, \cdots, n-1$.
Since $\Ext^n_R(R/\fb, M)$ is weakly Laskerian, $\Hom_R(R/\fb, H^n_\fb(M))$ is also weakly Laskerian by Lemma \ref{minimax}. If there exists $\fp \in \Supp_R(H^n_\fb(M))\subseteq V(\fb)$ with $\dim R/\fp=1$, then it is easy to see that $\dim M_\fp\leq n-1$ and so $(H^n_\fb(M))_\fp=0$ by Grothendieck's Vanishing Theorem, a contradiction.
Therefore, $$\Supp_R(H^n_\fb(M))\subseteq \Max(R).$$ This implies that the $R$-module $\Hom_R(R/\fb, H^n_\fb(M))$ is a weakly Laskerian $R$-module with support in $\Max(R)$. Hence, $H^n_\fb(M)$ is weakly Laskerian by Lemma \ref{wA} and so is $\fb$-weakly cofinite, as required.
\end{proof}

\begin{theo}\label{n7}
Let $M$ be an $R$-module of dimension $n$ such that $\Ext^i_R(R/\fa,M)$ is weakly Laskerian for all $i\leq n$. Then the $R$-module $\Ext^i_R(N,M)$ is weakly Laskerian for all $i\geq 0$ and for any finitely generated $R$-module $N$ with $\Supp_R(N) \subseteq V(\fa)$ and $\dim N \leq 1$.
\end{theo}
\begin{proof}
Let $N$ be a finitely generated $R$-module such that $\Supp_R(N) \subseteq V (\fa)$
and $\dim N \leq 1$. Then, using \cite[Theorem 6.4]{Matsu}, there exist prime ideals $\fp_1, \cdots, \fp_t$ of $R$ and a chain $0=N_0\subseteq N_1\subseteq\cdots\subseteq N_t=N$ of submodules of $N$ such that $N_j/N_{j-1}\cong R/\fp_j$ for all $j = 1, \cdots,t$. Since $\fp_j \in \Supp_R(N)$, we deduce that $\dim R/\fp_j \leq 1$ and so in the light of Proposition \ref{n6}, the $R$-module $H^i_{\fp_j}(M)$ is $\fp_j$-weakly cofinite for all $i \geq 0$ and for each $j = 1, \cdots, t$. Thus, by \cite[Corollary~3.10]{Mel}, the $R$-module $\Ext^i_R(R/\fp_j,M)$ is weakly Laskerian for all $i \geq 0$ and for each $j = 1, \cdots, t$. Now, considering the exact sequences
\begin{align*}
 0\rightarrow N_1 \rightarrow &N_2 \rightarrow R/\fp_2 \rightarrow 0\\
   0\rightarrow N_2 \rightarrow &N_3 \rightarrow R/\fp_3 \rightarrow 0 \\
    &\vdots \\
     0\rightarrow N_{t-1} \rightarrow &N_t \rightarrow R/\fp_t \rightarrow 0
\end{align*}
we infer that $\Ext^i_R(N, M)$ is weakly Laskerian, as desired.
\end{proof}


\providecommand{\bysame}{\leavevmode\hbox to3em{\hrulefill}\thinspace}
\providecommand{\MR}{\relax\ifhmode\unskip\space\fi MR }
\providecommand{\MRhref}[2]{%
  \href{http://www.ams.org/mathscinet-getitem?mr=#1}{#2}
}
\providecommand{\href}[2]{#2}

\end{document}